\documentclass[11pt,reqno]{amsart}
\usepackage{amsbsy,amsfonts,amsmath,amssymb,amscd,amsthm}
\usepackage{mathrsfs}
\usepackage{subfigure,enumitem,color,graphicx}
\usepackage[a4paper,text={6.1in,8in},centering,includefoot,foot=1in]{geometry}
\usepackage{pxfonts}
\usepackage[colorlinks=true,linkcolor=blue,citecolor=green]{hyperref}
\usepackage{srcltx}

\small\normalsize
\theoremstyle{plain}

\newtheorem{theorem}{Theorem}[section]
\newtheorem{lemma}[theorem]{Lemma}

\newtheorem{definition}[theorem]{Definition}

\newtheorem{remark}[theorem]{Remark}

\newtheorem{proposition}[theorem]{Proposition}

\newtheorem*{theorem A}{Theorem A}
\newtheorem*{theorem B}{Theorem B}
\newtheorem*{theorem C}{Theorem C}

\begin{document}
\title[Invariant graph and random bony attractors ]{Invariant graph and random bony attractors}
\author[Ghane]{F. H. Ghane$^*$}
\address{\centerline{Department of Mathematics, Ferdowsi University of Mashhad, }
\centerline{Mashhad, Iran.}  }
\email{ghane@math.um.ac.ir}
 \thanks{$^*$Corresponding author}
\author[Rabiee]{M.Rabiee}
\address{\centerline{Department of Mathematics, Ferdowsi University of Mashhad, }
\centerline{Mashhad, Iran.}  }
\email{maryam\_rabieefarahani@yahoo.com}
\author[Zaj]{M. Zaj}
\address{\centerline{Department of Mathematics, Ferdowsi University of Mashhad, }
\centerline{Mashhad, Iran.}  }
\email{zaj.marzie@yahoo.com}
\maketitle
\begin{abstract}
In this paper, we deal with random attractors for dynamical systems forced by a deterministic noise.
These kind of systems are modeled as skew products where the dynamics of the forcing process are described by the base transformation. Here, we consider
skew products over the Bernoulli shift with the unit interval fiber.
We study the geometric structure of maximal attractors, the orbit stability and stability of mixing of these skew products under random perturbations of the fiber maps.
We show that there exists an open set $\mathcal{U}$ in the space of such skew products so that any skew product
belonging to this set admits an attractor which is either a continuous invariant graph or a bony graph attractor.
These skew products have negative fiber Lyapunov exponents and their fiber maps are non-uniformly contracting, hence the non-uniform contraction rates
are measured by Lyapnnov exponents.
Furthermore, each skew product of $\mathcal{U}$ admits an invariant ergodic measure whose support is contained in that
attractor.
Additionally, we show that the invariant measure for the perturbed system is continuous in the Hutchinson metric.
\end{abstract}
\maketitle
\thispagestyle{empty}

\section{Introduction}
The qualitative study of dynamical systems is concerned with the study of attractors. Knowledge of the attractors may indicate the long time behavior
of the orbits. In the most simple cases, an attractor of a dynamical system is a union of finite set of smooth manifolds. There are interesting examples of locally dynamical systems having more complicated attractors.
For example 
in \cite{Ku}, Kudryashov introduced a new type of attractors so-called bony attractors, then
he presented an open set in the space of step skew products over the Bernoulli shift such that any of them had a bony attractor.
Following \cite{Ku}, an attractor $A$ of a skew product is \emph{bony} if $A$ is the union of
the graph of a continuous function on some subset of the base and an uncountable set
of vertical closed intervals (bones) contained in the closure of the graph.
This feature is similar to porcupine horseshoes discovered by Diaz and Gelfert in \cite{DG}. Indeed, from a topological point of view, a porcupine is a transitive set that looks like a horseshoe with infinitely many spines attached at various levels and in a dense way.

The objective of this article is to extend aforementioned result from \cite{Ku} to the random case, where the skew products are
general ( not necessarily step).
One novelty here is that, in our context, in contrast the Kudryashov' case, fiber maps are non-uniformly contracting, therefore
the contraction rates are non-uniform and hence measured by Lyapunov exponents.
We will discuss maximal attractors of this kind of skew
products and
show that they  are either a continuous invariant graph or a bony attractor.
Moreover, maximal attractors, carry an invariant ergodic measure
that projects to the Bernoulli measure in the base.

Notice that, in general, dynamical systems under the external forcing are modeled, in discrete time, as skew products,
\begin{equation}\label{00}
  F: \Omega \times M \to \Omega \times M, \ \ F(\omega,x)=(\theta \omega, f_\omega(x)),
\end{equation}
where the dynamics of the forcing process are described by the base transformation $\theta$ which is assumed to be a measure-preserving transformation of a probability space
$(\Omega, \mathcal{F}, \mathbb{P})$ (random forcing).
An \emph{invariant graph} of $F$ is the graph of a measurable function $\gamma: \Omega \to M$ which satisfies $f_\omega(\gamma(\omega))=\gamma(\theta(\omega))$, for $\mathbb{P}$-almost all $\omega \in \Omega$.

In the study of forced dynamical systems of the above form, invariant graphs play a central role since
they are the natural substitutes of a stable fixed point to the case of forced systems.
Furthermore, the existence of such invariant graphs considerably simplifies the dynamics of the forced systems.
Moreover, Lyapunov exponents yield additional information about the stability and attractivity of invariant graphs.
Attracting invariant graphs have a wide variety of applications in many branches of nonlinear
dynamics (e.g. \cite{CD, DC, HOJ1, HOJ2, PC, SD, T} etc.).
 A context in which the attractivity of invariant graphs plays a central role is
generalised synchronisation, a phenomenon that has been widely studied in
theoretical physics.
In \cite{St2} Stark provides the conditions for the existence and regularity of invariant graphs
and discusses a number of applications. His results include some generalizations to the case of
non-uniform contraction.
We mention that in skew product systems with uniformly contracting
fiber maps, there exist continuous invariant attracting sets for the overall
dynamics, see \cite{HP}, Theorem 6.1a, \cite{HPS}.
Results in the non-uniform case, when the fiber map possesses negative Lyapunov exponents in
the fibre \cite{AC1, E1, E2, Z1, Z2}, are very recent and invariant graphs are very sensitive to perturbations.\\
\textbf{This work is organized as follows:}
In Subsections 1.1 and 1.2 we recall some standard definitions.
Then we state our main result in Subsection 1.3.
The proof of our main result, Theorem A below, is given in Section 2.
\subsection{Preliminaries}
Assume that $X$ is a metric measure space.
Denote by $\textnormal{int}(D)$ and $\text{Cl}(D)$, respectively, the interior and the closure of any set $D$.

Let $(X; \mathcal{B}; \mu; f)$ be a measure preserving dynamical system.
If $f$ is invertible then, based on \cite{Br, VO}, the system is  {\it Bernoulli} if it is isomorphic to a Bernoulli shift.
Clearly invertible systems cannot be isomorphic to non-invertible systems. But there
is a construction to make a non-invertible system invertible, namely by passing to the
natural extension.
For  non-invertible case, being Bernoulli means that the natural extension is isomorphic to a Bernoulli shift.

 The map $f$ is \emph{mixing} (or \emph{strong mixing}) if
$$\mu(f^{-n}(A) \cap B)\to \mu(A)\mu(B), \ as \ n \to +\infty,$$
for every $A,B \in \mathcal{B}$.
Every mixing system \cite{VO} is necessarily ergodic.

For a metric space $X$, putting
$$\textnormal{Lip}_1(X) = \{f : X \to \mathbb{R} : |f(x)- f(y)| \leq d(x, y) \ \textnormal{for} \ \textnormal{all} \ x, y \in X\},$$
define the \emph{Hutchinson metric} on the set $\mathcal{M}(X)$, the space of all Borel probability measures,
by
\begin{equation}\label{e20}
d_H(\nu,\mu) = \sup \{ |\int_X f d\nu - \int_X f d\mu : f \in \textnormal{Lip}_1(X)|\}.
\end{equation}
In \cite[Thm.~3.1]{K}, the author proved that for every metric space $X$, the topology $\mathcal{T}$ on $\mathcal{M}(X)$ generated by $d_H(\nu,\mu)$
coincides with the topology $\mathcal{W}$ of weak convergence if and only if $\textnormal{diam}(X)< \infty$.
Moreover, the space $\mathcal{M}(X)$ is complete in the metric $d_H$ if and only if $X$ is complete (see \cite[Thm.~4.2]{K}).

The concept of a weak contraction map was introduced in 1997 by Alber and
Guerre-Delabriere \cite{AG}.
We say that a continuous map $f$ is \emph{weak contraction} (or \emph{distance decreasing } \cite{Bi}) whenever for each $x,y \in X$ with $x\neq y$, $d(f(x),f(y)) < d(x,y)$.\\
It is a well-known fact \cite[Coro.~3]{J} (see also \cite{Bi}) that if $f$ is weak contraction and $X$ is compact then there exists a unique fixed point
$x \in X$ of the map $f$. Furthermore, for every $y \in X$, $\lim_{k \to \infty}f^k(y)=x$ uniformly.
Then we say that $x$ is a \emph{weak attracting fixed point}.
Clearly if $f$ is a weak contraction map then
$$d(f^{n}(y), f^{n}(z)) \to 0, \ as, \ n \to \infty,$$
for each $y, z \in X$.
\subsection{Random maps and skew products}
A \emph{random map} with base $(\Omega, \mathcal{F}, \mathbb{P}, \theta)$, in the sense of Arnold \cite{Ar}, is a skew product of the
form (\ref{00}) where $(\Omega, \mathcal{F}, \mathbb{P})$ is a probability space, $\theta : \Omega \to \Omega$ is a bi-measurable and ergodic
measure-preserving bijection and $M$ is a measurable space. If $M$ is a smooth manifold and all fibre
maps $f_\omega$ are $C^r$, we call $F$ a random $C^r$-map.

Take $\Sigma^+_k = \{0,\dots, k-1 \}^\mathbb{N}$ and $\Sigma_k = \{0,\dots, k-1 \}^\mathbb{Z}$ endowed with the product topology and
equip them with the Bernoulli measures $\nu^+$ and $\nu$, respectively, corresponding to some distribution of probabilities $p_0, \ldots, p_{k-1}$, which gives us the probability with which we apply $f_i$.
Here, assume that the probabilities $p_i$, $i=0, \ldots, k-1$, are the same and equal to $1/k$.
Let $\sigma:\Sigma_k \to \Sigma_k$ and $\sigma^+:\Sigma_k^+ \to \Sigma_k^+$ denote the one-sided and two-sided left shift.
It is well known that \cite{W2} $\sigma^+$ and $\sigma$ are ergodic transformations preserving the probabilities $\nu^+$ and $\nu$, respectively.

Let $M$ be a compact smooth manifold. Here, we consider skew products of the form
\begin{equation}
 F:\Sigma_k \times M \to \Sigma_k \times M; \ \
  (\omega,x) \to (\sigma \omega,f_{w}(x))
\label{e11}
\end{equation}
which is called a \emph{skew product over the Bernoulli shift}, where $\omega \in \Sigma_k$, $x \in M$ and the maps $f_\omega$ are $C^r$ diffeomorphisms on $M$. The space $\Sigma_k$
 is called the \emph{base}, the space $M$ is called the \emph{fiber}, and the maps $f_\omega$ are called the \emph{fiber maps}. Thus each skew product of the form (\ref{e11}) is a
 random $C^r$-map.

A skew product over the Bernoulli shift is a \emph{step skew product} if
the fiber maps $f_\omega$ depend only on the digit $\omega_0$ and not on the whole
sequence $\omega$. We emphasise, in contrast to step skew products, the fiber maps of (general) skew products of the form (\ref{e11})
depend on the whole sequence $\omega$.
When treating a step skew product for one sided time $\mathbb{N}$, this results in the skew product system $F^+$ on $\Sigma_k^+ \times M$:
\begin{equation}\label{s}
 F^+:\Sigma_k^+ \times M \to \Sigma_k^+ \times M; \ \
  (\omega,x) \to (\sigma^+ \omega,f_{w_0}(x)).
\end{equation}

We denote iterates of a skew product system $F$ of the form (\ref{e11}) as $F^n(\omega,x)=(\sigma^n(\omega), f_\omega^n(x))$. Here, for
$n \geq 1$
$$f_\omega^n(x):=f_{\sigma^{n-1}\omega} \circ \ldots \circ f_\omega(x).$$
For a step skew product system this becomes
$$f_\omega^n(x):=f_{\omega_{n-1}} \circ \ldots \circ f_{\omega_0}(x),$$
where $\omega=(\ldots, \omega_{-1}, \omega_0, \omega_1, \ldots, \omega_n, \ldots)\in \Sigma_k$.

In the rest of this article we assume that the fiber $M$ is always the unit interval $I$.

Take $\mathcal{C}(I)$ the space of all random $C^2$-maps (general skew products) acting on $\Sigma_k \times I$ defined by $C^2$ interval diffeomorphisms.
We equip $\mathcal{C}(I)$
 with the following metric:
\begin{equation}\label{e113}
  \textnormal{dist}_{C^2}(F,G):=\sup_{\omega \in \Sigma_k}(\textnormal{dist}_{C^2}(f_\omega^{\pm 1} , g_\omega^{\pm 1})), \ \textnormal{for \ each} \ F,G \in \mathcal{C}(I),
\end{equation}
where $f_\omega$ and $g_\omega$ are the fiber maps of $F$ and $G$, respectively.

Let $F: \Sigma_k \times I \to \Sigma_k \times I$ be a homeomorphism onto its image, but suppose its image is
contained strictly in $\Sigma_k \times I$. The (global) \emph{maximal attractor} of $F$ is defined as:
\begin{equation}\label{e18}
 A_{max}(F):=\bigcap_{n=0}^\infty F^n(\Sigma_k \times I).
\end{equation}
\subsection{Main results}
To state the main result precisely, the concept of a bony attractor may need to be introduced.
\begin{definition}\label{bony}
Following \cite{Ku}, an attractor $\Lambda$ of a skew product $F$ is a bony graph attractor if $\Lambda$ is the union of
the graph of a continuous function $\gamma$ defined on some set of full measure of the base and a set
of vertical closed intervals ("bones") contained in the closure of the graph.
\end{definition}
In this article, we will show that maximal attractors of a certain class of general skew products (random maps) are either a continuous invariant graph or a bony attractor.
Our novelty here is that the fiber maps of such systems depend on the whole sequence $\omega$ and hence
they are not necessarily step skew products.
Moreover, the fiber contraction rates are non-uniform
and hence measured by Lyapunov exponents, in addition, the attractors carry an ergodic measure.
Our result thus extends work by Kudryashov in \cite{Ku} who treated step skew products over the Bernoulli shift having
bony attractors.
\begin{theorem A}
There exists an open nonempty set $\mathcal{U}$ in the space $C^2$ random maps $\mathcal{C}(I)$ given by (\ref{e113}) such that any system $G$ belonging to this
set has a maximal attractor $A_{max}(G)$ satisfies the following properties:
\begin{enumerate}
  \item the maximal attractor $A_{max}(G)$
 is either a continuous invariant graph or a bony graph attractor;
 \item there exists an invariant ergodic measure $\mu_G$ whose support is the closure of the graph $\Gamma_G$, in particular, $(G, \Gamma_G, \mu_G)$ is Bernoulli and therefore it is mixing,
 additionally, the invariant measure for the perturbed system is continuous in the Hutchinson metric;
 \item the fiber Lyapunov exponent of $G$ is negative;
\end{enumerate}
Moreover, the set of random maps of $\mathcal{U}$ which admit a bony graph attractor is nonempty.
\end{theorem A}
 \section{Proof of Theorem A}
To prove Theorem A, we provide a single step skew product $F$, and show that
every skew product (random map) $G\in \mathcal{C}(I)$ (not necessarily step) which is close enough to $F$ (with respect to $\textnormal{dist}_{C^2}$ given by (\ref{e113})) satisfies the conclusion of the theorem.

Consider the interval $I = [0, 1]$ and let $\{f_0, \ldots, f_{k-1}\}$ a finite family of orientation preserving
(strictly increasing) $C^2$-diffeomorphisms defined on $I$ enjoying the following conditions:
\begin{enumerate}
  \item [(a)] The mappings $f_i$, $i=0, \ldots, k-1$, bring the unit interval $I$ strictly into itself and they are $C^2$ close to the identity.
  \item [(b)] $f_0$ is a weak contraction with a unique weak attracting fixed point $p_0$, i.e. $Df_0(p_0)=1$.
  \item [(c)] $f_i$, $i=1, \ldots, k-1$, are uniformly contracting maps such that any of them has a unique attracting fixed point $p_i$.
   \item [(d)] We have "contraction on average", i.e. for each $x \in I$, $\prod_{i=0}^{k-1} Df_i(x) <1$.
  \item [(e)] The fixed points $p_i$, $i=0, \ldots, k-1$, are pairwise disjoint, $p_i\neq 0,1$ and satisfy the no-cycle condition, i.e. $f_i(p_j)\neq p_k$ for each distinct indices $i,j$ and $k$.
   \item [(f)] Let $p_0< p_1<\ldots <p_{k-1}$ and $J = [p_0 , p_1]$. Assume we have "covering property", i.e. there exists the points $x_0$ and $x_1$ such that $p_0 < x_0 < x_1 < p_1$ and the interval $B = (x_0, x_1) \subset \text{int}(J)$ for which the following holds:
$\forall x\in [x_0,x_1], \ \|Df_i(x)\|<1 \ \text{and} \   \text{Cl}(B) \subset f_0(B) \cup f_1(B).$
\end{enumerate}
Write the step skew product
\begin{equation}
 F:\Sigma_k \times I \to \Sigma_k \times I; \ \
  (\omega,x) \to (\sigma \omega,f_{w_0}(x))
\label{e66}
\end{equation}
whose fiber maps are the mapping $f_i$, $i=0, \ldots, k-1$, which satisfy the properties listed above. Fix the skew product $F$ and take a small open ball $\mathcal{U}$
around $F$ in the space $\mathcal{C}(I)$. To prove Theorem A, we show that any system $G$
belonging to this set satisfies the conclusion of the theorem.

We define the \emph{Transfer Operator} $T : \mathcal{M}(I) \to \mathcal{M}(I)$ by the formula,
$$T(\mu)(B) := \frac{1}{k}\sum_{i=1}^k \mu(f_i^{-1}(B)),$$ for any Borel subset $B$ and for each measure $\mu \in  \mathcal{M}(I)$, where $\mathcal{M}(I)$ is the space of all Borel probability measures on $I$. If a measure $\mu \in \mathcal{M}(I)$ is
a fixed point of the transfer operator we say that $\mu$ is a \emph{stationary measure}.
\begin{remark}\label{rems}
The following two facts hold:
\begin{enumerate}
  \item The contraction on average condition given by $(d)$ ensuring \cite{S} the existence of a unique attractive stationary probability measure $m$ in the
sense that $T^n \mu$ converges weakly to $m$, for any probability measure $\mu \in \mathcal{M}(I)$.
  \item For the skew product $F^+$ of the form (\ref{s}) with the fiber maps $f_i$, $i=0, \ldots, k-1$, the product measure $\nu^+ \times m$ is an ergodic invariant measure.
The skew product $F$ given by (\ref{e66}) is the natural extension of $F^+$.
Invariant measures for $F^+$ with marginal $\nu^+$ and invariant measures for $F$ with
marginal $\nu$ are in one to one relationship, as detailed in \cite{Ar}.
A stationary measure $m$ thus, through the invariant measure $\nu^+ \times m$ for $F^+$,
gives rise to an invariant measure $\mu$ for $F$, with marginal $\nu$.
\end{enumerate}
\end{remark}
\subsection{Fiber Lyapunov exponents}
For each Lipschitz map $f : I \to I$ we define the norm $\|.\|$ by
\begin{equation}\label{e011}
\|f\|:=\sup_{x \neq x^{\prime}}\frac{|f(x)-f(x^{\prime})|}{|x-x^{\prime}|}.
\end{equation}
It is easily seen that, whenever $f$ is $C^1$, by Mean value theorem for real-valued functions,
$$\|f\|=\sup \{\|Df(x)\|: x\in I\}.$$
For the skew product $F(\omega,x)=(\sigma \omega,f(\omega,x))=(\sigma \omega,f_{\omega_0}(x))$ given by (\ref{e66}), consider a sequence of functions $\varphi^n$ defined by $\varphi^n(\omega)=\|f^n(\omega,.)\|$.
It is simply verified that the family of functions $\{a_n\}$ defined by $a_n(\omega)=\log (\varphi^n(\omega))$
is subadditive.
By definition of the mappings $f_i$, $i=0, \ldots, k-1$, and the functions $\varphi^n$ one has that $\log^+ (\varphi^1) \in L^1(\nu)$, hence
by Kingman's Subadditive Theorem \cite[Thm.~3.3.3]{VO}, the limit
\begin{equation}\label{L}
 \lambda(\omega):=\lim_{n \to \infty}\frac{1}{n} \log \|f^n(\omega,.)\|
\end{equation}
exists at $\nu$-almost every point. Moreover, the function $\lambda \in  L^1(\nu)$ and
\begin{equation}\label{inf}
\lim_{n \to \infty}\frac{1}{n} \log \|f^n(\omega,.)\|=\inf_n \frac{1}{n} \log \|f^n(\omega,.)\|.
\end{equation}
By ergodicity of $\nu$, the limit (\ref{L}) is constant, denoted by $\lambda$. The contraction on average property given by condition $(d)$ ensures that $\lambda$ is negative.
 The constant $\lambda$ is called the \emph{m-fiber Lyapunov exponent} with respect to $\nu$.
 \begin{lemma}\label{cor0}
 There exists an open subset $\mathcal{U}\subset \mathcal{C}(I)$ containing $F$ such that any skew product $G$ belonging to this set admits a negative $m$-fiber Lyapunov exponents with respect to $\nu$.
 \end{lemma}
\begin{proof}
Take small neighborhoods $U_i \subset \textnormal{Diff}^2(I)$ of the fiber maps $f_i$, $i=0, \ldots, k-1$, of $F$ and let $\mathcal{U}\subset \mathcal{C}(I)$ a small open neighborhood of $F$ enjoying the following property:
 there exists a constant $C>0$ such that for any $G \in \mathcal{U}$ with $G(\omega,x)=(\sigma \omega,g(\omega,x))=(\sigma \omega,g_\omega(x))$ one has that
 \begin{equation}\label{c2}
 \ \forall \omega \in \Sigma_k, \ \text{the \ map} \ g_\omega \in U_{\omega_0}, \  \text{and} \
\textnormal{dist}_{C^2}(g_\omega, g_{\omega^{\prime}})< C d(\omega, \omega^{\prime}), \ \text{for \ any} \ \omega, \omega^{\prime}\in \Sigma_k.
 \end{equation}
Then, by this fact and (\ref{inf}), for given a sufficiently small $\varepsilon > 0$ there exists $\delta > 0$ such that if $\text{diam}(U_i)< \delta$ then
\begin{equation}\label{lyapunov}
\lim_{n \to \infty}\frac{1}{n}\textnormal{log} \| g^n(\omega,.)\|=\lambda + \varepsilon < 0, \ \textnormal{for} \ \nu \ a.e. \ \omega \in \Sigma_k.
\end{equation}
In particular, $G$ possesses a negative $m$-fiber Lyapunov exponent.
\end{proof}
\subsection{Maximal attractors and invariant graphs}
For the step skew product $F$ given by (\ref{e66}) and any general skew product $G \in \mathcal{U}$, consider the maximal attractors $ A_{max}(F)$ and $ A_{max}(G)$, respectively, defined by
\begin{equation}\label{e67}
  A_{max}(F):=\bigcap_{n\geq 0}F^n(\Sigma_k \times I), \ \ A_{max}(G):=\bigcap_{n\geq 0}G^n(\Sigma_k \times I).
\end{equation}
A first main step in the proof of Theorem A is to show that the attractor is an invariant graph.
For that, we get the next proposition which is an analogue of \cite[Thm.~5]{C2}, \cite[Thm.~1.4]{St2} and \cite[Pro.~2.3]{BHN} to our setting.
\begin{proposition}\label{thm000}
Consider the skew product $F$ given by (\ref{e66}). For each general skew product $G \in \mathcal{U}$, given by Lemma \ref{cor0}, there exists a measurable function $\gamma_G:\Omega \subseteq \Sigma_k \to I$, with $\nu(\Omega)=1$ such that $\Gamma_G$ the graph of $\gamma_G$ is invariant under $G$.
The closure of the graph $\Gamma_G$ is the support of an invariant  ergodic measure $\mu_G$, in particular, $(G,\Gamma_G,\mu_G)$ is Bernoulli and hence it is mixing.
Furthermore, $\Gamma_G$ is attracting in the sense that for $ \omega \in \Omega$, $\lim_{n \to \infty}|\pi_x(G^n(\omega,x))- \gamma_G(\sigma^n \omega)|=0$ for every $x \in I$, where $\pi_x$ is the natural projection from $\Sigma_k \times I$ to $I$.
\end{proposition}
\begin{proof}
By Lemma \ref{cor0}, each $G \in \mathcal{U}$ has negative $m$-fiber Lyapunov exponent.
By (\ref{lyapunov}), given $\varepsilon > 0$ there exists a measurable function $C : \Sigma_k \to \mathbb{R}^+$ such that for $\nu$ a.e. $\omega \in \Sigma_k$, we have
\begin{equation}\label{3.7}
 \| g^n(\omega,.)\| < C(\omega)e^{(\lambda + \varepsilon)n}, \ \textnormal{for} \ \textnormal{all} \ n> 0.
\end{equation}
Since the Bernoulli shift $\sigma$ is ergodic and invertible hence
$\sigma^{-1}$ is ergodic with respect to $\nu$ and has the same spectrum of Lyapunov exponents by Furstenberg-Kesten Theorem \cite{FK}. Thus if we define
$$h_n(\omega,x):=g^n(\sigma^{-n}\omega,x)$$
then by (\ref{lyapunov})
$$  \lim_{n \to \infty}\frac{1}{n}\textnormal{log} \| h_n(\omega,.)\|=\lambda + \varepsilon < 0, \ \textnormal{for} \ \nu \ a.e. \ \omega \in \Sigma_k.$$
Hence there exists $\ell(\omega)$ such that
$$\| h_n(\omega,.)\|<e^{n(\lambda + \varepsilon)} \ \ \forall n\geq \ell(\omega).$$
Thus given $\varepsilon > 0$ there exists a measurable function $C : \Sigma_k \to \mathbb{R}^+$ such that for $\nu$ a.e. $\omega \in \Sigma_k$, we have
\begin{equation}\label{3.7}
 \| h_n(\omega,.)\| < C(\omega)e^{(\lambda + \varepsilon)n}, \ \textnormal{for} \ \textnormal{all} \ n> 0.
 \end{equation}
Applying the approach used in the proof of \cite[Pro.~2.3]{BHN}, we conclude that the sequence $\{h_\ell(\omega,x)\}$ is a
Cauchy sequence for every $x \in I$ and a.e. $\omega \in \Sigma_k$.
Indeed, let $$\alpha(x):=\sup_{\omega \in \Sigma_k}|x-g(\omega,x)|$$
and note that for $x$ fixed $\alpha(x)$ is finite as $\Sigma_k$ is compact and $g$ is continuous. Given any $\varepsilon^{\prime}>0$, choose $\ell^{\ast}(\omega)$ sufficiently large that
$$\alpha(x)C(\omega)\sum_{j=\ell^{\ast}(\omega)}^\infty e^{j(\lambda + \varepsilon)}< \varepsilon^{\prime}. $$
Then if $m >\ell>\ell^{\ast}(\omega)$
$$|h_m(\omega,x)-h_\ell(\omega,x)| \leq \alpha(x)\sum_{j=\ell}^\infty \|h_j(\omega,.) \| \leq \alpha(x)c(\omega)\sum_{j=\ell}^\infty e^{j(\lambda + \varepsilon)}< \varepsilon^{\prime}.$$
To see this note that
$$|h_m(\omega,x)-h_\ell(\omega,x)|=|h_m(\omega,x)-h_{m-1}(\omega,x)+ \ldots +h_{\ell +1}(\omega,x)-h_\ell(\omega,x)|.$$
Note that applying $G$ once to $(\sigma^{-k}(\omega),x)$, gives $G(\sigma^{-k}(\omega),x)=(\sigma^{-(k-1)}(\omega),g(\sigma^{-k}(\omega),x))$.
Thus, $h_k(\omega,x)-h_{k-1}(\omega,x)=h_{k-1}(\omega,g(\sigma^{-k}(\omega),x)))- h_{k-1}(\omega,x)$.
As a result, $$|h_{k-1}(\omega,x)-h_{k}(\omega,x)| \leq \|h_{k-1}(\omega,.) \| |x- g(\sigma^{-k}(\omega),x))|. $$ Hence
\begin{quote}
$|h_m(\omega,x)-h_\ell(\omega,x)| \leq |h_m(\omega,x)-h_{m-1}(\omega,x)|+ \ldots +|h_{\ell +1}(\omega,x)-h_\ell(\omega,x)|$
\end{quote}
\begin{quote}
 $=\sum_{j=\ell+1}^m | h_j(\omega,x)-h_{j-1}(\omega,x)|$
\end{quote}
\begin{quote}
$\leq \sum_{j=\ell+1}^m \| h_{j-1}(\omega,.) \| |x-g(\sigma^{-j}(\omega),x))|$
\end{quote}
\begin{quote}
$\leq \sum_{j=\ell+1}^\infty \| h_{j-1}(\omega,.) \| \alpha(x).$
\end{quote}
Thus
\begin{quote}
$|h_m(\omega,x)-h_\ell(\omega,x)| \leq  \alpha(x) \sum_{j=\ell+1}^\infty \| h_{j-1}(\omega,.) \|$
\end{quote}
\begin{quote}
 $=\alpha(x)C(\omega)\sum_{j=\ell+1}^\infty e^{(j-1)(\lambda + \varepsilon)}< \varepsilon^{\prime}$
\end{quote}
as $\ell > \ell^{\ast}(\omega)$. Thus there exists a subset $\Omega \subseteq \Sigma_k$, with $\nu(\Omega)=1$, so that for each $\omega \in \Omega$ the sequence $\{h_m(\omega,x)\}$ is a Cauchy sequence for every $x \in I$.
Define
\begin{equation}\label{graph}
\gamma_G: \Omega \to I, \ \gamma_G(\omega):= \lim_{n \to + \infty}h_n(\omega,0).
\end{equation}
Since $$G(\omega, h_\ell(\omega,0))=(\sigma \omega,h_{\ell +1}(\sigma \omega,0)), $$
we see that
$$G(\omega, \gamma_G(\omega))=(\sigma \omega,\gamma_G(\sigma \omega) ) $$
and hence $\Gamma_G$, the graph of $\gamma_G$ is invariant under $G$.
Furthermore, by construction, for every $\omega \in \Omega$, one has
\begin{equation}\label{lim}
\lim_{n \to +\infty}|g^n(\omega,x)-g^n(\omega,\gamma_G(\omega))|=\lim_{n \to +\infty}|g^n(\omega,x)-g^n(\omega,0)|=0.
\end{equation}
This is because
$$|g^n(\omega,x)-g^n(\omega,0)| \leq \|g^n(\omega,.)\| |x|$$
and $\|g^n(\omega,.)\| \to 0$ as $n \to +\infty$.

Therefore, for every $\omega \in \Omega$,
\begin{equation}\label{almost}
 \lim_{n \to + \infty}g_{\sigma^{-1}\omega} \circ \ldots \circ g_{\sigma^{-n}\omega}(I)= \lim_{n \to + \infty}g_{\sigma^{-1}\omega} \circ \ldots \circ g_{\sigma^{-n}\omega}(0)=\gamma_G(\omega).
\end{equation}
Hence, $\gamma_G$ induces an invariant graph for $G$ which is an attracting set by (\ref{almost}).

Consider the projection $p_G:\Gamma_G \to \Sigma_k$, $p_G(\omega,\gamma_G(\omega))=\omega$, which is an isomorphism onto its image and the measure
\begin{equation}\label{measure}
 \mu_G=(p_G)_\ast \nu=\nu \circ (id \times \gamma_G)^{-1}.
\end{equation}
Then the mixing properties of the base transformation $(\sigma, \Sigma_k, \nu)$ lift to the transformation $(G, \Gamma_G\Gamma, \mu_G)$. In particular, $\mu_G$ is Bernoulli which implies that it is mixing. Every system that is mixing is also ergodic. Hence $\mu_G$ is an ergodic measure.
\end{proof}
We now point out that the previous proposition with together the next two results establish assertions $(1)$ and $(2)$ of the main result of this article, Theorem A.
\begin{proposition}\label{graph1}
For each skew product $G \in \mathcal{U}$ the maximal attractor $A_{max}(G)$ is either a continuous invariant graph or a bony attractor.
\end{proposition}
\begin{proof}
Take a skew product $G \in \mathcal{U}$ with $G(\omega,x)=(\sigma \omega, g(\omega,x))=(\sigma \omega, g_\omega(x))$. By the previous proposition there exists a measurable function $\gamma_G:\Omega \subseteq \Sigma_k \to I$, with $\nu(\Omega)=1$, such that $\Gamma_G$ the graph of $\gamma_G$ is invariant under $G$.
We claim that $\Gamma_G\subset A_{max}(G).$

Indeed, since $A_\omega:=A_{max}(G) \cap I_\omega=\bigcap_{n\geq 0} I(\omega,n), \ \textnormal{where} \ I(\omega,n):= g_{\sigma^{-1}\omega}\circ \ldots \circ g_{\sigma^{-n}\omega}(I)$ and $I_\omega:=\{\omega\}\times I$, and by using (\ref{almost}), one has
$$\lim_{n \to +\infty} g_{\sigma^{-1}(\omega)}\circ \dots \circ g_{\sigma^{-n}(\omega)}(I)=\lim_{n \to +\infty} g_{\sigma^{-1}(\omega)}\circ \dots \circ g_{\sigma^{-n}(\omega)}(0)=\gamma_G(\omega), $$ for each $\omega \in \Omega$, hence we observe that
$\Gamma_G\subset A_{max}(G)$, as claimed.

Note that $I(\omega,n)$ is a sequence of nested intervals, and thus $A_\omega=A_{max}(G) \cap I_\omega$ is either an interval or a
single point.
Also note that if some sequences $\omega$ and $\omega^{\prime}$ are close enough to each other,
say, $$\omega_{-n}^{\prime}=\omega_{-n}, \ldots, \omega_{-1}^{\prime}=\omega_{-1}$$
then, using $I(\omega^{\prime},n)\supset A_{\omega^{\prime}}$, we deduce $I(\omega,n)\supset A_{\omega^{\prime}}.$
This implies the upper-semicontinuity of $A_\omega$. This semicontinuity,
 will immediately imply the continuity of its graph part.

Now there are two possibilities:
either $\Omega=\Sigma_k$ and hence $A_{max}(G)$ is a continuous invariant graph, or the bones exist.
In the later case, to verify that $A_{max}(G)$ is actually a bony attractor, it is enough
to show that the set of bones contained in the closure of the graph. This will be done in the following lemma which completes the proof of the proposition.
\end{proof}
\begin{lemma}\label{geometry}
Let $G \in \mathcal{U}$ be a small perturbation of the skew product $F$ given by (\ref{e66}) such that its maximal attractor $A_{max}(G)$ contains the bones with a graph function $\gamma_G$ defined on a full measure subset $\Omega \subset \Sigma_k$. Then
the bones are contained in the closure of the graph $\Gamma_G$.
\end{lemma}
\begin{proof}
To prove the lemma it is enough to show that the maximal attractor $A_{max}(G)$ coincides with the
closure of the intersection $A_{max}(G) \cap (\Omega \times I)$.

First, we notice that the fiber maps $f_i$, $i=1, \ldots, k-1$, of $F$ are uniformly contracting maps, by condition $(c)$, and the skew product $G$ is $C^2$-close to $F$,
hence, by $(\ref{c2})$, every sequence $\omega \in \Sigma_k$ without a tail of $0$'s to the left
 belong to $\Omega$.
Assume $(\omega,x)\in A_{max}(G)$ with $\omega \in \Sigma_k \setminus \Omega$. Then the sequence $\omega$ has a tail of $0$'s to the left
(i.e. there exists $n_0 \in \mathbb{N}$ so that for each $n >n_0$, one has $\omega_{-n}=0$).
We denote the set of sequences $\omega^{\prime}$ such that 
$\omega_i=\omega_i^{\prime}$ for $i \in [-N,N]$ by
$U_N(\omega)$ and the $\varepsilon$-neighborhood of the point $x$ by $V_\varepsilon(x)$. Take $n> n_0> N$ with $n=n_0+2m$ for large enough $m$ and $h=g^{-1}_{\sigma^{-(n_0+m)}\omega}\circ \ldots \circ g^{-1}_{\sigma^{-1}\omega}$.
Note that $\sigma^{-(n_0+m)}\omega$ has the following form
$$ \sigma^{-(n_0+m)}\omega=(\ldots, 0, \ldots,0;\underbrace{0, \ldots,0}_{m-times}, \omega_{-n_0}, \ldots, \omega_{-1}, \omega_0, \omega_1, \omega_0, \ldots).$$
Then the point $(\omega^{\prime},x^{\prime})=(\sigma^{-(n_0+m)}(\omega),h(x))\in A_{max}(G)$.
Now we take the sequence $$\widetilde{\omega}=(\ldots, 1, 1, \underbrace{0, \ldots, 0}_{2m-times}, \omega_{-n_0}, \ldots, \omega_{-1}; \omega_0, \omega_1, \ldots)$$
which has a tail of $1$'s to the left. Since $f_1$ is a uniformly contracting map, $G$ is $C^2$-close to $F$ and by $(\ref{c2})$, we conclude that
$\text{diam}(g_{\widetilde{\omega}_{-n}}, \ldots, g_{\widetilde{\omega}_{-1}}(I))\to 0$ whenever $n \to +\infty$. Thus $\widetilde{\omega}\in \Omega$.
Moreover, for $0 < \delta < \varepsilon <1$ small and large enough $k$, one has $\text{diam}(g_{\sigma^{-n_0-2m-k}\widetilde{\omega}}\circ \ldots \circ g_{\sigma^{-n_0-2m-2k}\widetilde{\omega}}(I))<\delta < \varepsilon$. Let us take $I_\delta=g_{\sigma^{-n_0-2m-k}\widetilde{\omega}}\circ \ldots \circ g_{\sigma^{-n_0-2m-2k}\widetilde{\omega}}(I)$. Then for large enough $m$, we have
$$g_{\sigma^{-n_0-m-1}\widetilde{\omega}}\circ \ldots \circ g_{\sigma^{-n_0-2m}\widetilde{\omega}}(I_\delta)\subset h(V_\varepsilon(x)).$$
Thus, the pair $(\widetilde{\omega}, \gamma_G(\widetilde{\omega}))$ belongs to the intersection $A_{max}(G) \cap (\Omega \times I)\cap (U_N(\omega)\times V_\varepsilon(x))$ and hence the conclusion of the lemma holds.
\end{proof}
The next result ensures that the subset of all skew products $G \in \mathcal{U}$ having a bony attractor is nonempty.
\begin{lemma}\label{graph2}
There exists a small perturbation $G \in \mathcal{U}$ of $F$ which admits a bony graph attractor in the sense of Definition \ref{bony}. In particular, the subset of bones has the cardinality of the continuum and is dense in the attractor.
\end{lemma}
\begin{proof}
Consider the fiber map $f_0$ satisfies conditions ($a$) and ($b$) in the beginning of Section 2 with a weak attracting fixed point $p_0$, and take a map $g$, $C^2$-close to $f_0$, such that $g=id$ on a small neighborhood $U$ of the point $p_0$.
Now take a small perturbation $G$ of $F$ so that for the sequence $\omega=(\ldots, 0, 0, 0, \ldots)\in \Sigma_k$ one has $g_\omega=g$.
As you have seen before, for $I_\omega=\{\omega\}\times I$,
and $I(\omega,n)=g_{\sigma^{-1}(\omega)}\circ \dots \circ g_{\sigma^{-n}(\omega)}(I)$,
one has
$A_{max}(G) \bigcap I_\omega = \bigcap_{n\geq 0} I(\omega,n).$
Thus, we get $$I(\omega,n)=g_{\sigma^{-1}(\omega)}\circ \dots \circ g_{\sigma^{-n}(\omega)}(I)=\underbrace{g \circ \ldots \circ g}_{(n)-times}(I)=g^{n}(I)$$
which ensures that $A_{max}(G) \bigcap I_\omega = \bigcap_{n\geq 0} I(\omega,n)$ is an interval, hence $A_{max}(G)$ is a bony attractor.
Moreover, for each sequence $\omega^{\prime} \in \Sigma_k$ of the form $\omega^{\prime}=(\ldots, 0, 0; \omega^{\prime}_1, \omega^{\prime}_2, \ldots)$, it is not hard to see that $A_{max}(G) \bigcap I_{\omega^{\prime}}$
is an interval.

Furthermore, by construction, for any finite word $\alpha$ of the alphabets $\{0, 1, \ldots, k-1\}$ and a sequence $\rho$ of the form $\rho=(\ldots, 0, 0, \alpha, 0, 0, \ldots)$ with $\alpha$ standing at the zero position, $A_{max}(G) \bigcap I_\rho$ contains an interval. Thus
the subset of bones has the cardinality of the continuum and is dense in the attractor $A_{max}(G)$.
\end{proof}
In what follows, we show that the maximal attractor $A_{max}(F)$ is thick, this means that the projection of $A_{max}(F)$ on the fiber has positive Lebesgue measure.
\begin{lemma}\label{lem11}
Consider the skew product $F$ given by (\ref{e66}). Then the maximal attractor $A_{max}(F)$ is thick.
\end{lemma}
\begin{proof}
First, we recall conditions $(a)-(f)$ in the beginning of this section. By Proposition \ref{thm000} and conditions $(b)$ and $(c)$, the maximal attractor $A_{max}(F)$ is a continuous invariant graph. Consider the graph function
$\gamma_F:\Sigma_k \to I$ and let $K=\gamma_F(\Sigma_k)$. We apply the covering property from condition $(f)$ and show that the interval $B$ with $B \subset J=[p_0,p_1]$ introduced by $(f)$ is contained in $K$.
By this fact, Remark \ref{rems}, the definition of measure $\mu_F$ given by (\ref{measure}) and by construction, the maximal attractor $A_{max}(F)$ is thick.
For that, we show that for each $x \in B$, there exists a sequence $(\omega_{-n})_{n \geq 1}$ of $\{0, 1\}$ so that
\begin{equation}\label{34}
  x=\lim_{n \to + \infty}f_{\omega_{-1}}\circ \ldots \circ f_{\omega_{-n}}(I).
\end{equation}
First, we define, inductively, a sequence $(\omega_{-n})_{n \geq 1}$ of $\{0, 1\}$ so that
\begin{equation}\label{334}
  x=\lim_{n \to + \infty}f_{\omega_{-1}}\circ \ldots \circ f_{\omega_{-n}}(B).
\end{equation}
Assume that we have found $\omega_{-1}, \ldots, \omega_{-n} \in \{0, 1\}$ so that $x \in f_{\omega_{-1}}\circ \ldots \circ f_{\omega_{-n}}(B)$. Then the covering property implies that
$$ x \in f_{\omega_{-1}}\circ \ldots \circ f_{\omega_{-n}}(B) \subset \bigcup_{i=0}^1 f_{\omega_{-1}}\circ \ldots \circ f_{\omega_{-n}} \circ f_i(B),$$
hence we can find $\omega_{-(n+1)}$ such that $x \in f_{\omega_{-1}}\circ \ldots \circ f_{\omega_{-n}}\circ f_{\omega_{-(n+1)}}(B).$
Take any sequence $\omega^{\prime}\in \Sigma_k$ so that for each $n \geq 1$, we have $\omega^{\prime}_{-n}=\omega_{-n}$.
Then it is easily seen that $$x=\lim_{n \to \infty}f_{\omega_{-1}}\circ \ldots \circ f_{\omega_{-n}}(B)=\lim_{n \to \infty}f_{\omega_{-1}}\circ \ldots \circ f_{\omega_{-n}}(I)=\lim_{n \to \infty}f_{\omega^{\prime}_{-1}}\circ \ldots \circ f_{\omega^{\prime}_{-n}}(I),$$ as we claimed.
\end{proof}
The next proposition is an analogue of \cite[Thm.~3.1]{BHN} to our setting. It asserts that the invariant measure for the perturbed system is
continuous in the Hutchinson metric.
\begin{proposition}\label{pro12}
Suppose $G \in \mathcal{U}$ with $G(\omega,x)=(\sigma\omega,g(\omega,x))$.
Then for given $\varepsilon >0$, by shrinking $\mathcal{U}$, for $\nu$ almost every $\omega\in \Sigma_k$ and all $x\in I$, one has that
$$d(F^n(\omega,x),G^n(\omega,x))<\varepsilon,$$
except for at most a fraction $\varepsilon$ of times $n$, where the distance $d$ between two points of $\Sigma_k \times I$ is the sum of the distances between their projections onto
the base and onto the fiber.

Furthermore,  $d_H(\mu_F,\mu_G)< \varepsilon$, where $d_H$ is Hutchinson metric given by (\ref{e20}) and $\mu_F$ and $\mu_G$ are the measures obtained from Proposition \ref{thm000} for $F$ and $G$.
\end{proposition}

\end{document}